\batchmode
\makeatletter
\def\input@path{{"/Users/russw/Documents/Research/mypapers/Partitionability is NP-complete/"}}
\makeatother
\documentclass[12pt,oneside,english]{amsart}
\usepackage[T1]{fontenc}
\usepackage[latin9]{inputenc}
\usepackage{babel}
\usepackage{array}
\usepackage{url}
\usepackage{enumitem}
\usepackage{amstext}
\usepackage{amsthm}
\usepackage{amssymb}
\usepackage{geometry}
\usepackage[pdftex,pdfusetitle,
 bookmarks=true,bookmarksnumbered=false,bookmarksopen=false,
 breaklinks=false,pdfborder={0 0 1},backref=false,colorlinks=false]
 {hyperref}

\makeatletter

\providecommand{\tabularnewline}{\\}

\numberwithin{equation}{section}
\numberwithin{figure}{section}
\theoremstyle{plain}
\newtheorem{thm}{\protect\theoremname}[section]
\theoremstyle{plain}
\newtheorem{cor}[thm]{\protect\corollaryname}
\theoremstyle{plain}
\newtheorem{lem}[thm]{\protect\lemmaname}
\theoremstyle{plain}
\newtheorem{prop}[thm]{\protect\propositionname}
\theoremstyle{remark}
\newtheorem{rem}[thm]{\protect\remarkname}
\theoremstyle{plain}
\newtheorem{question}[thm]{\protect\questionname}

\usepackage{lmodern}
\usepackage{tikz}
\usepackage{subfig}

\newcommand{\prob}[1]{\textsc{#1}}
\newcommand{\Partitionability}{\prob{Partitionability}}

\makeatother

\providecommand{\corollaryname}{Corollary}
\providecommand{\lemmaname}{Lemma}
\providecommand{\propositionname}{Proposition}
\providecommand{\questionname}{Question}
\providecommand{\remarkname}{Remark}
\providecommand{\theoremname}{Theorem}

\begin{document}
\global\long\def\cc{\mathbb{C}}%

\global\long\def\ff{\mathbb{F}}%

\global\long\def\pp{\mathbb{P}}%

\global\long\def\rr{\mathbb{R}}%

\global\long\def\rp{\mathbb{R}\mathrm{P}}%

\global\long\def\poincare{\mathrm{P}}%

\global\long\def\zz{\mathbb{Z}}%

\global\long\def\bdry{\partial}%

\global\long\def\link{\operatorname{link}}%

\global\long\def\del{\operatorname{del}}%

\global\long\def\cone{\operatorname{Cone}}%

\global\long\def\redeul{\tilde{\chi}}%

\global\long\def\Epsilon{\mathrm{E}}%

\global\long\def\Tau{\mathrm{T}}%

\global\long\def\Kappa{\mathrm{K}}%

% Shared TikZ styles for the partitionability gadget figures.
% Needs: \usepackage{tikz}
\usetikzlibrary{calc,arrows.meta}

\definecolor{ink}{HTML}{262626}
\definecolor{relred}{HTML}{B8312F}    % faces of the relative part A
\definecolor{varblue}{HTML}{2B5BAA}   % variable (output) vertices
\definecolor{facetfill}{HTML}{F4F1EA} % pale paper tint for triangles
\definecolor{burn}{HTML}{F2B33D}      % highlighter for the faces of an interval
\definecolor{arrowink}{HTML}{8F5414}  % arrows R_F -> F

% layers: facet fills at the back, then burn bands, then ink
\pgfdeclarelayer{fills}\pgfdeclarelayer{burns}
\pgfsetlayers{fills,burns,main}

\tikzset{
  every picture/.append style={line cap=round, line join=round},
  e/.style={draw=ink, line width=0.6pt},                         % ordinary edge
  rel/.style={draw=relred, line width=0.8pt},                    % edge in A
  pin/.style={draw=ink, line width=0.5pt,                        % pinned edge
    double=white, double distance=1.4pt},
  spine/.style={draw=ink, line width=1.8pt},                     % spine of a 3-page book
  free/.style={e, dash pattern=on 2.4pt off 1.8pt},              % interface edge left to a neighbour
  % vertices (white ring so edges stop just short of the dot)
  v/.style={circle, inner sep=0pt, minimum size=5.4pt, fill=ink, draw=white, line width=0.8pt},
  vA/.style={v, fill=relred},                                    % vertex in A
  vout/.style={v, fill=varblue},                                 % output vertex (variable gadget)
  voutc/.style={circle, inner sep=0pt, minimum size=4.6pt,       % output vertex seen from a clause
    fill=white, draw=varblue, line width=0.85pt},
  vused/.style={voutc, minimum size=5.4pt, fill=burn!90!facetfill},                  % output vertex taken by a clause-side interval
  lab/.style={font=\small, inner sep=1.5pt},
  glab/.style={font=\scriptsize, text=ink!55, inner sep=1pt},
  cov/.style={-{Stealth[round, length=3.6pt, width=3pt]}, line width=0.65pt, draw=arrowink},
  burnband/.style={draw=burn, line width=3pt, opacity=0.5, line cap=round, line join=round},
}

% switches: \burnfalse turns the highlighter bands off, \arrowsfalse the arrows
\newif\ifburn \burntrue
\newif\ifarrows \arrowstrue

% \taperarrow{start}{end}{r0}{r1}: the curve  start to[bend left=22] end, drawn as a dense run
% of dots whose radius goes from r0 to r1 (pt), finished with a hand-shaped head along the tangent.
% (Bezier control points of "bend left=22": angle +-22 degrees, distance 0.3915*chord.)
\makeatletter
\newcommand{\taperarrow}[4]{%
  \begingroup
  \coordinate (tp@A) at (#1); \coordinate (tp@B) at (#2);
  \pgfpointanchor{tp@A}{center}\pgfgetlastxy{\tp@ax}{\tp@ay}%
  \pgfpointanchor{tp@B}{center}\pgfgetlastxy{\tp@bx}{\tp@by}%
  \pgfmathsetmacro\tp@ax{\tp@ax}\pgfmathsetmacro\tp@ay{\tp@ay}%
  \pgfmathsetmacro\tp@bx{\tp@bx}\pgfmathsetmacro\tp@by{\tp@by}%
  \pgfmathsetmacro\tp@dx{\tp@bx-\tp@ax}\pgfmathsetmacro\tp@dy{\tp@by-\tp@ay}%
  \pgfmathsetmacro\tp@px{\tp@ax+0.3915*(0.92718*\tp@dx-0.37461*\tp@dy)}%
  \pgfmathsetmacro\tp@py{\tp@ay+0.3915*(0.37461*\tp@dx+0.92718*\tp@dy)}%
  \pgfmathsetmacro\tp@qx{\tp@bx-0.3915*(0.92718*\tp@dx+0.37461*\tp@dy)}%
  \pgfmathsetmacro\tp@qy{\tp@by-0.3915*(-0.37461*\tp@dx+0.92718*\tp@dy)}%
  \pgfmathsetmacro\tp@L{1.04*veclen(\tp@dx,\tp@dy)}%
  \pgfmathsetmacro\tp@te{max(0.3,1-2.8/\tp@L)}%
  \foreach \tp@i in {0,...,48} {%
    \pgfmathsetmacro\tp@t{\tp@i/48*\tp@te}%
    \pgfmathsetmacro\tp@s{1-\tp@t}%
    \pgfmathsetmacro\tp@x{\tp@s^3*\tp@ax+3*\tp@s^2*\tp@t*\tp@px+3*\tp@s*\tp@t^2*\tp@qx+\tp@t^3*\tp@bx}%
    \pgfmathsetmacro\tp@y{\tp@s^3*\tp@ay+3*\tp@s^2*\tp@t*\tp@py+3*\tp@s*\tp@t^2*\tp@qy+\tp@t^3*\tp@by}%
    \pgfmathsetmacro\tp@r{#3+(#4-#3)*\tp@i/48}%
    \fill[arrowink] (\tp@x pt,\tp@y pt) circle[radius=\tp@r pt];}%
  \pgfmathsetmacro\tp@ang{atan2(\tp@by-\tp@qy,\tp@bx-\tp@qx)}%
  \fill[arrowink, shift={(\tp@bx pt,\tp@by pt)}, rotate=\tp@ang]
    (0.3pt,0) -- (-3.9pt,2.1pt) .. controls (-3pt,0.6pt) .. (-2.9pt,0)
    .. controls (-3pt,-0.6pt) .. (-3.9pt,-2.1pt) -- cycle;
  \endgroup}
\makeatother

% arrow style: 0 = plain curved, 1 = tapered thin->thick, 2 = "comet" (fat at R_F, thin at head)
\def\covmode{0}
\newcommand{\covarrow}[2]{%
  \ifcase\covmode
    \draw[cov] (#1) to[bend left=22] (#2);
  \or \taperarrow{#1}{#2}{0.1}{1.0}%
  \or \taperarrow{#1}{#2}{1.25}{0.35}%
  \fi}

% \facet{a}{b}{c}: pale fill of a triangle (on the back layer)
\newcommand{\facet}[3]{%
  \begin{pgfonlayer}{fills}\fill[facetfill] (#1) -- (#2) -- (#3) -- cycle;\end{pgfonlayer}}

% \vcov{v}{a}{b}: the facet vab has restriction face {v}; its interval is {v, va, vb, vab}.
\newcommand{\vcov}[3]{%
  \begingroup
  \coordinate (vc@c) at (barycentric cs:#1=1,#2=1,#3=1);
  \coordinate (vc@v) at ($(#1)!0.17!(vc@c)$);
  \coordinate (vc@a) at ($(#2)!0.17!(vc@c)$);
  \coordinate (vc@b) at ($(#3)!0.17!(vc@c)$);
  \ifburn\begin{pgfonlayer}{burns}
    \draw[burnband] ($(vc@v)!0.86!(vc@a)$) -- (vc@v) -- ($(vc@v)!0.86!(vc@b)$);
  \end{pgfonlayer}\fi
  \ifarrows\covarrow{$(#1)!0.24!(vc@c)$}{$(#1)!0.9!(vc@c)$}\fi
  \endgroup}
% \ecov{a}{b}{c}: the facet abc has restriction face ab; its interval is {ab, abc}.
\newcommand{\ecov}[3]{%
  \begingroup
  \coordinate (vc@c) at (barycentric cs:#1=1,#2=1,#3=1);
  \coordinate (vc@m) at ($(#1)!0.5!(#2)$);
  \coordinate (vc@a) at ($(#1)!0.17!(vc@c)$);
  \coordinate (vc@b) at ($(#2)!0.17!(vc@c)$);
  \ifburn\begin{pgfonlayer}{burns}
    \draw[burnband] ($(vc@a)!0.12!(vc@b)$) -- ($(vc@a)!0.88!(vc@b)$);
  \end{pgfonlayer}\fi
  \ifarrows\covarrow{$(vc@m)!0.14!(vc@c)$}{$(vc@m)!0.95!(vc@c)$}\fi
  \endgroup}

% \copycoords{name}{x}: coordinates of a copy gadget (square wheel) with bottom-left corner at (x,0).
% Rim order 1,2,3,4 = bottom-left, top-left, top-right, bottom-right; hub h.  Rim edges 23 and 41 are pinned.
\newcommand{\copycoords}[2]{%
  \coordinate (#1-1) at (#2,0);
  \coordinate (#1-2) at ($(#2,0)+(0.38,1.65)$);
  \coordinate (#1-3) at ($(#2,0)+(1.62,1.65)$);
  \coordinate (#1-4) at ($(#2,0)+(2,0)$);
  \coordinate (#1-h) at ($(#2,0)+(1,0.78)$);}
% \copyedges{name}: draw its edges
\newcommand{\copyedges}[1]{%
  \facet{#1-h}{#1-1}{#1-2}\facet{#1-h}{#1-2}{#1-3}\facet{#1-h}{#1-3}{#1-4}\facet{#1-h}{#1-4}{#1-1}
  \foreach \i in {1,2,3,4} {\draw[e] (#1-h) -- (#1-\i);}
  \draw[e] (#1-1) -- (#1-2);
  \draw[e] (#1-3) -- (#1-4);
  \draw[pin] (#1-2) -- (#1-3);
  \draw[pin] (#1-4) -- (#1-1);}

% Coordinates and base drawing of the positive clause Cl_1 = pqr with its three converters.
% Converter i: Conv(u_i; slot) with outer vertex w_i.  pqr is shrunk a little toward the centre
% (factor 0.8) so the converter corners have room for labels.
\newcommand{\clausecoords}{%
  \coordinate (w1) at (90:2.5);  \coordinate (w2) at (210:2.5); \coordinate (w3) at (330:2.5);
  \coordinate (p) at ($(0,0)!0.8!($(w1)!0.5!(w2)$)$);
  \coordinate (q) at ($(0,0)!0.8!($(w2)!0.5!(w3)$)$);
  \coordinate (r) at ($(0,0)!0.8!($(w3)!0.5!(w1)$)$);
  \coordinate (u1) at (barycentric cs:w1=1,p=1,r=1);
  \coordinate (u2) at (barycentric cs:w2=1,p=1,q=1);
  \coordinate (u3) at (barycentric cs:w3=1,q=1,r=1);}
\newcommand{\clausebase}{%
  \facet{p}{q}{r}
  \foreach \u/\a/\b/\c in {u1/w1/p/r, u2/w2/p/q, u3/w3/q/r} {\facet{\u}{\a}{\b}\facet{\u}{\b}{\c}\facet{\u}{\c}{\a}}
  \foreach \a/\b in {w1/p,p/w2,w2/q,q/w3,w3/r,r/w1} {\draw[rel] (\a) -- (\b);}
  \foreach \a/\b in {p/q,q/r,r/p} {\draw[e] (\a) -- (\b);}
  \foreach \u/\a/\b/\c in {u1/w1/p/r, u2/w2/p/q, u3/w3/q/r} {
    \draw[e] (\u) -- (\a); \draw[e] (\u) -- (\b); \draw[e] (\u) -- (\c);}
  \foreach \x in {w1,w2,w3,p,q,r} {\node[vA] at (\x) {};}}

\title{Partitionability is NP-complete}
\author{Russ Woodroofe}
\address{Univerza na Primorskem, Glagoljaška 8, 6000 Koper, Slovenia}
\email{russ.woodroofe@famnit.upr.si}
\urladdr{\url{https://woodroofe.net/}}
\begin{abstract}
We show that the problem of deciding whether a simplicial complex
is partitionable is NP-complete in dimensions two and higher, answering
a 1995 question of Kleinschmidt and Onn.
\end{abstract}

\maketitle

\section{\label{sec:Introduction}Introduction}

A simplicial complex $\Delta$ is \emph{partitionable} if we may partition
its set of faces into parts in such a way that each part consists
of the faces between some facet $\sigma$ and a face $\varphi(\sigma)\subseteq\sigma$.
Thus, if $[\varphi(\sigma),\sigma]=\left\{ \tau:\varphi(\sigma)\subseteq\tau\subseteq\sigma\right\} $,
then the required condition is that 
\[
\Delta=\dot{\bigcup_{\sigma\text{ facet}}}[\varphi(\sigma),\sigma].
\]
 We call the map $\varphi$ a \emph{partitioning scheme} or more shortly
a \emph{partitioning} of $\Delta$. We immediately see that a partitioning
scheme makes it convenient to count the faces of $\Delta$. Indeed,
if $\Delta$ is pure $d$-dimensional, then the generating function
of the numbers $f_{k}$ of faces of cardinality $k$ satisfies 
\begin{equation}
f_{0}+f_{1}x+\cdots+f_{d+1}x^{d+1}=h_{0}(1+x)^{d+1}+h_{1}x(1+x)^{d}+\cdots+h_{d+1}x^{d+1}(1+x)^{0},\label{eq:fvechvec}
\end{equation}
where $h_{k}$ is the number of facets $\sigma$ so that $\varphi(\sigma)$
has cardinality $k$. 

Partitionability first arose in the theses of Ball and of Provan,
see \cite{Ball/Provan:1982}. It is important in the theory of shellability
\cite{Stanley:1996}; indeed, one definition of a shelling is as a
partitioning scheme with an ordering of the parts that satisfies an
additional condition. There has been some recent interest in partitionable
simplicial complexes \cite{Bergeron/Dermenjian/Machacek:2020,Doolittle/Goeckner/Lazar:2022,Duval/Goeckner/Klivans/Martin:2016,Hachimori:2021,Juhnke-Kubitzke/Venturello:2019,Santamaria-Galvis:2022}.

The decision problem \Partitionability{} asks whether a simplicial
complex $\Delta$ (given by a list of facets) is partitionable. Kleinschmidt
and Onn conjectured in \cite[Section 5]{Kleinschmidt/Onn:1995} (see
also \cite{Kleinschmidt/Onn:1996}) that \Partitionability{} is NP-complete
in dimensions two and higher, and moreover that the problem remains
NP-complete when restricted to pseudomanifolds. Noble showed that
\Partitionability{} is in NP even when dimension is not restricted
\cite{Noble:1996}. Kaibel and Pfetsch included the \Partitionability{}
problem as the final item in their 2002 list of algorithmic problems
in polytope theory \cite[Problem 35]{Kaibel/Pfetsch:2003}, and it
was the last remaining simplicial complex problem on their list whose
computational complexity was unknown in fixed dimension.

In this paper, I will settle this problem for general simplicial complexes. 
\begin{thm}
\label{thm:Main}\Partitionability{} is NP-complete, even when restricted
to pure \linebreak{}
$2$-dimensional simplicial complexes.
\end{thm}

It follows easily that the same holds for $d\geq2$. On the other
hand, the problem admits a polynomial algorithm in dimension one \cite[Problem 35]{Kaibel/Pfetsch:2003}.
\begin{cor}
\label{cor:MainDimD}Let $d\geq2$. \Partitionability{} is NP-complete,
even when restricted to pure $d$-dimensional simplicial complexes.
\end{cor}

The proof here does not follow the approach suggested in \cite{Kleinschmidt/Onn:1995}.
Indeed, a computer search for a ``switching complex'' (see Remark~\ref{rem:SwitchingComplex})
in the sense of Kleinschmidt and Onn was not successful. I will instead
give a reduction directly from \prob{1-In-3-Sat}. A main ingredient
over the approach of Kleinschmidt and Onn will be control of the Euler
characteristic of the simplicial complex produced by the reduction.

Partitionability is frequently studied in conjunction with shellable
simplicial complexes. Indeed, any shellable complex is partitionable.
The \prob{Shellability} problem asks whether a simplicial complex
(given by a list of facets) is shellable. \prob{Shellability} was
shown to be NP-complete by Goaoc, Paták, Patáková, Tancer, and Wagner
\cite{Goaoc/Patak/Patakova/Tancer/Wagner:2019} (see also \cite{Patak/Tancer:2023,Santamaria-Galvis/Woodroofe:2021}).
The simplicial complexes used for the reduction in \cite{Goaoc/Patak/Patakova/Tancer/Wagner:2019}
or in \cite{Santamaria-Galvis/Woodroofe:2021} are built by taking
a $2$-sphere for each variable, and wrapping clause gadgets partway
around the equators. Truth assignments correspond to (homology facets
in) the north/south hemispheres of the variable spheres. The complexes
are Cohen-Macaulay by \cite[Lemma 2.2]{Duval/Goeckner/Klivans/Martin:2016},
and we sketch a proof in Section~\ref{sec:Remarks} that the construction
of \cite{Santamaria-Galvis/Woodroofe:2021} is partitionable. In contrast,
the simplicial complexes that we will build will have trivial Euler
characteristic, but highly nontrivial homology; thus, they will never
be Cohen-Macaulay.

The \emph{Stanley depth} of a simplicial complex is a numeric parameter
arising from commutative algebra considerations \cite[Section 5]{Stanley:1982},
see \cite{Herzog/Jahan/Yassemi:2008}. It extends partitionability
in the sense that a pure $d$-dimensional simplicial complex is partitionable
if and only if it has Stanley depth of $d+1$ (equivalently, Stanley
depth $\geq d+1$) \cite[Section 4]{Herzog/Jahan/Yassemi:2008}. The
associated \prob{Stanley Depth} decision problem asks, given a simplicial
complex $\Delta$ and a parameter $k$, is the Stanley depth of $\Delta$
at least $k$?
\begin{cor}
Let $d\geq2$. \prob{Stanley Depth} is NP-hard, even when restricted
to pure $d$-dimensional simplicial complexes.
\end{cor}

Herzog, Vladoiu, and Zheng, working in the algebraic context of monomial
ideals, show in \cite{Herzog/Vladoiu/Zheng:2009} that Stanley depth
is computable, but remark that this does not give a practical algorithm.
It follows from their work or from \cite{Herzog/Jahan/Yassemi:2008}
that \prob{Stanley Depth} is in NP (and hence NP-complete) for bounded-dimensional
simplicial complexes. Ichim, Katthän, and Moyano-Fernández, working
in the broader algebraic context of multigraded modules, give an algorithm
in \cite{Ichim/Katthan/Moyano-Fernandez:2017} for computing Stanley
depth, but make no particular claim of efficiency.

\subsection*{Organization}

The paper is organized as follows. In Section~\ref{sec:Preliminaries},
we give background material on simplicial complexes and partitionability.
In Section~\ref{sec:Reduction}, we construct the necessary simplicial
complex gadgets and describe the reduction. In Section~\ref{sec:PartitionWGluing},
we introduce the notion of a partitioning with gluing, used to localize
and combine partitionings. In Section~\ref{sec:MainProofs}, we prove
Theorem~\ref{thm:Main}. In Section~\ref{sec:Remarks}, we make
additional remarks and discuss open questions.

\subsection*{AI disclosure}

A precursor of the $V^{0}$ gadget was discovered in a computer search
carried out by Claude Opus 5.5, and the argument was then developed
and refined in conjunction with Opus and Claude Fable 5.1 models.
Figures were drawn by hand and converted to TikZ by Opus. The paper
is human-written.

\section*{Acknowledgements}

Discussing partitionability with Andrés Santamaría-Galvis has been
enlightening over the years. Bennet Goeckner (among other helpful
comments) suggested highlighting NP-completeness for \prob{Shellability}
even given partitionability. I thank the University of Michigan Department
of Mathematics for its hospitality during the 2026--2027 academic
year.

\section{\label{sec:Preliminaries}Preliminaries}

We assume that the reader has general familiarity with NP-completeness,
on the level of \cite{Papadimitriou:1994} or similar. We review basics
of simplicial complexes.

A \emph{simplicial complex} is simply a hereditary set system: a family
of sets $\Delta$ (called \emph{faces}) so that if $\sigma\in\Delta$
and $\tau\subseteq\sigma$, then also $\tau\in\Delta$. Simplicial
complexes are important in discrete geometry, as we may identify a
singleton set with a vertex, a doubleton with an edge, a $3$-element
set with a triangle, and so forth. A simplicial complex is \emph{pure}
\emph{$d$-dimensional} if all of its \emph{facets} (maximal faces)
have cardinality $d+1$, and so \emph{dimension} $d$. All simplicial
complexes discussed in this paper are finite.

The main setting that we focus on is a pure $2$-dimensional complex,
which we see consists of a set of triangles, together with all edges
and vertices that are contained in a triangle. The empty set is a
face of every non-trivial simplicial complex.

A \emph{relative simplicial complex} is a pair $(\Delta,\Sigma)$,
where $\Delta$ is a simplicial complex and $\Sigma$ is a simplicial
complex with $\Sigma\subseteq\Delta$. As a set system, we consider
$(\Delta,\Sigma)$ as consisting of the faces of $\Delta$ that are
not in $\Sigma$. Thus, a relative simplicial complex is partitionable
if the set system $(\Delta,\Sigma)$ may be partitioned as in the
introduction: 
\[
(\Delta,\Sigma)=\dot{\bigcup_{\sigma\,\text{facet}}}[\varphi(\sigma),\sigma].
\]
 Equation (\ref{eq:fvechvec}) goes through without change. 

We will occasionally need to distinguish between the simplicial complexes
$\Sigma=\emptyset$ (with no faces) and $\Sigma=\{\emptyset\}$ (having
the empty set as its unique face).

The \emph{reduced Euler characteristic} of the simplicial complex
$\Delta$ is given by 
\[
\tilde{\chi}(\Delta):=-f_{0}+f_{1}-f_{2}+f_{3}-\cdots,
\]
where $f_{i}$ is the number of faces in $\Delta$ of cardinality
$i$. Thus, the reduced Euler characteristic includes the empty set
$f_{0}$ in the alternating sum of face numbers. Since we have $f_{0}=1$
except in the degenerate case, we see that $\tilde{\chi}(\Delta)$
is one smaller than the (unreduced) Euler characteristic $\chi(\Delta)$.
The \emph{(reduced) Euler characteristic} of a relative simplicial
complex $(\Delta,\Sigma)$ has the same definition, where $f_{i}$
is the number of faces of $\Delta$ that are not faces of $\Sigma$.
Thus, $\chi(\Delta,\emptyset)=\tilde{\chi}(\Delta)$ while $\chi(\Delta,\Sigma)=\tilde{\chi}(\Delta)-\tilde{\chi}(\Sigma)=\chi(\Delta)-\chi(\Sigma)$
when $\Sigma\neq\emptyset$. 

We record that the $h_{d+1}$ entry in equation (\ref{eq:fvechvec})
has a nice description:
\begin{lem}
\label{lem:PartEulerChar}If $(\Delta,\Sigma)$ is a pure $d$-dimensional
relative complex with $f_{i}$ and $h_{i}$ as in equation (\ref{eq:fvechvec}),
then $h_{d+1}=(-1)^{d}\chi(\Delta,\Sigma)$.
\end{lem}

\begin{proof}
Plug in $x=-1$.
\end{proof}
In particular, if $\chi(\Delta,\Sigma)=0$ and $\varphi$ is a partitioning
of $(\Delta,\Sigma)$, then there is no facet $\sigma$ with $\varphi(\sigma)=\sigma$. 

Passing from $2$-dimensional complexes to higher-dimensional complexes
is not difficult. The \emph{cone} of the simplicial complex $\Delta$
is the simplicial complex $\cone\Delta$ which adds one new \emph{apex}
vertex $a$ to the vertex set of $\Delta$, and whose facets are $\left\{ \sigma\cup a:\sigma\,\text{facet of }\Delta\right\} $.
I did not find the following lemma in the literature in this form,
although \cite[Lemma 3.6]{Herzog/Vladoiu/Zheng:2009} may be viewed
as an algebraic/Stanley depth generalization.
\begin{lem}
If $\Delta$ is a pure $d$-dimensional complex, then $\Delta$ is
partitionable if and only if $\cone\Delta$ is partitionable.
\end{lem}

\begin{proof}
Let $a$ be the apex vertex of $\cone\Delta$. If $\varphi$ is a
partitioning of $\Delta$, then the map sending $\sigma\cup a$ to
$\varphi(\sigma)$ is a partitioning of $\cone\Delta$. 

Conversely, let $\varphi$ be a partitioning of $\cone\Delta$. If
$\sigma$ is a facet of $\Delta$, then the face $\sigma$ of $\cone\Delta$
lies only in the facet $\sigma\cup a$, hence $\varphi(\sigma\cup a)\subseteq\sigma$.
In particular, the map sending $\sigma$ to $\varphi(\sigma\cup a)$
is a well-defined map from facets of $\Delta$ to faces of $\Delta$,
and is easily seen to be a partitioning of $\Delta$.
\end{proof}
In particular, Corollary~\ref{cor:MainDimD} now follows immediately
from Theorem~\ref{thm:Main}.

\section{\label{sec:Reduction}The reduction}

We provide a polynomial time reduction from the \prob{1-In-3-Sat}
problem \cite{Schaefer:1978} (see also \cite[LO4]{Garey/Johnson:1979}).
This problem takes as input a Boolean formula in the form 
\[
(\ell_{1}\vee\ell_{2}\vee\ell_{3})\wedge\cdots\wedge(\ell_{3m-2}\vee\ell_{3m-1}\vee\ell_{3m}),
\]
where each \emph{literal} $\ell_{i}$ is either a Boolean variable
$x$ or its negation $\neg x$, and asks whether there is an assignment
of truth values so that each clause $(\ell_{3k+1}\vee\ell_{3k+2}\vee\ell_{3k+3})$
has exactly one true literal (equivalently, exactly two false literals).
It is convenient to assume that each variable $x$ appears in at least
two literals (including in negated form). This assumption causes no
loss of generality, as if $x$ occurs only once, we simply repeat
the clause in which it appears. See Remark~\ref{rem:DupedVars} below.

\subsection{Gadgets}

As is usual in NP-completeness reductions from \prob{3-Sat} and its
variants, the proof proceeds by building gadgets to represent variables
and clauses in the simplicial complex setting. We will start by showing
that \Partitionability{} is NP-complete when expanded to relative
simplicial complexes, and then show that we may restrict back to non-relative
complexes.

We have three main gadgets: a variable gadget $V(k)$, a positive
clause gadget $C^{+}$, and a negated clause gadget $C^{\neg}$, as
we will now describe.

The variable gadget $V^{0}=V(2)$ is formed in two stages. We start
by subdividing a square into four triangles $12r,23r,34r,14r$. We
then attach the boundary of a tetrahedron to the edge $23$, and another
to the edge $14$. We treat the vertex $r$ as being relative. Thus,
$V^{0}$ has $9$ vertices (1 of which is relative), $18$ edges,
and $12$ triangles.

\begin{figure}
% Variable gadget: a chain of copy gadgets, vertex 4 of each identified with vertex 1 of the next.
\begin{tikzpicture}[scale=1]
  \copycoords{a}{0}\copycoords{b}{2}\copycoords{c}{4}\copycoords{z}{7.3}
  \copyedges{a}\copyedges{b}\copyedges{c}\copyedges{z}
  \node at ($(c-4)!0.5!(z-1)+(0,0.8)$) {$\cdots$};
  % hubs (in A)
  \foreach \n in {a,b,c,z} {\node[vA] at (\n-h) {};}
  % chain vertices
  \foreach \p in {b-1,c-1,c-4,z-1} {\node[v] at (\p) {};}
  % outputs
  \foreach \n in {a,b,c,z} {
    \node[vout,label={[lab]above:$\neg x$}] at (\n-2) {};
    \node[vout,label={[lab]above:$x$}] at (\n-3) {};}
  \node[vout,label={[lab]below:$x$}] at (a-1) {};
  \node[vout,label={[lab]below:$\neg x$}] at (z-4) {};
\end{tikzpicture}

\caption{\label{fig:VarGadget}Variable gadget $V(k)$. Each double line edge
has an independent tetrahedral boundary attached. Vertices labeled
$x$ or $\neg x$ (shown in blue) are variable vertices.}
\end{figure}
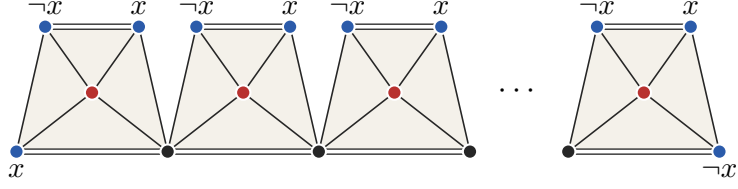

The variable gadget $V(k)$ is obtained from $k-1$ copies of $V^{0}$
by gluing the vertex $4$ from the $i$th copy to the vertex $1$
of the $(i+1)$st copy. See Figure~\ref{fig:VarGadget}. The \emph{positive
variable vertices} in $V(k)$ consist of the vertex $3$ from all
copies of $V^{0}$ and the (unglued) vertex $1$ from the first copy
of $V^{0}$. The \emph{negated variable vertices} in $V(k)$ consist
of the vertex $2$ from all copies of $V^{0}$ and the (unglued) vertex
$4$ from the last copy of $V^{0}$. Thus, $V(k)$ has $k$ positive
variable vertices and $k$ negated variable vertices. The \emph{relativizable
vertices} (shown in red) consist of the vertex $r$ from all copies
of $V^{0}$.

The positive clause gadget $C^{+}$ is pictured in Figure~\ref{fig:ClausePosGadget}~(left).
It consists of three subdivided triangles, each glued to an edge of
a center triangle; see also Section~\ref{subsec:PartitioningGadgets}.
The \emph{hollow vertices} will later be glued to positive or negated
variable vertices. We treat the boundary of $C^{+}$ (shown in red,
and generated by the edges that are only in a single facet) as being
relative.

The negated clause gadget $C^{\neg}$ is formed by gluing three copies
of the complex pictured in Figure~\ref{fig:ClauseNegGadget}~(right)
along the bolded edge. The hollow vertices will later be glued to
positive or negated variable vertices. We treat the boundary of $C^{\neg}$
(shown in red, generated by edges that are in a single facet) as being
relative.

\begin{figure}[h]
\begin{tabular}[b]{rl>{\raggedright}p{1cm}rl}
 &
% The 1-in-3 clause Cl_1 = pqr with its three converters (one per slot).
\begin{tikzpicture}[scale=1.15]
  \clausecoords\clausebase
  \node[voutc,label={[lab,label distance=1pt]270:$\ell_1$}] at (u1) {};
  \node[voutc,label={[lab,label distance=1pt]20:$\ell_2$}] at (u2) {};
  \node[voutc,label={[lab,label distance=1pt]160:$\ell_3$}] at (u3) {};
\end{tikzpicture} &
 &
 &
% One page of the 2-in-3 clause Cl_2 (page s t a, spine st bold) with the converter on its slot s a.
% Page and converter are equilateral triangles forming a rhombus.
\begin{tikzpicture}[scale=1.15]
  \coordinate (t) at (0,0); \coordinate (s) at (0,1.6);
  \coordinate (a) at (-1.386,0.8); \coordinate (w) at (-1.386,2.4);
  \coordinate (u) at (barycentric cs:w=1,a=1,s=1);
  \facet{t}{a}{s}\facet{u}{a}{s}\facet{u}{s}{w}\facet{u}{w}{a}
  \draw[rel] (t) -- (a); \draw[rel] (a) -- (w); \draw[rel] (w) -- (s);
  \draw[e] (a) -- (s);
  \draw[e] (u) -- (a); \draw[e] (u) -- (s); \draw[e] (u) -- (w);
  \draw[spine] (s) -- (t);
  \foreach \x in {t,s,a,w} {\node[vA] at (\x) {};}
  \node[lab,anchor=west] at ($(s)+(0.08,0)$) {$s$};
  \node[lab,anchor=west] at ($(t)+(0.08,0)$) {$t$};
  \node[lab,anchor=east] (lu) at ($(u)+(-0.95,0.15)$) {$\neg\ell_i$};
  \draw[ink!45, line width=0.4pt, shorten >=3.5pt] (lu.east) -- (u);
  \node[voutc] at (u) {};
\end{tikzpicture}\tabularnewline
\end{tabular}

\caption{Clause gadgets.\protect \\
\label{fig:ClausePosGadget}\emph{Left}: Positive clause gadget $C^{+}$.
Hollow vertices $\ell_{1},\ell_{2},\ell_{3}$ will be glued to $x$
or $\neg x$ variable vertices.\protect \\
\emph{Right}: \label{fig:ClauseNegGadget}One third of negated clause
gadget $C^{\neg}$. The full gadget is formed by gluing three copies
of the shown complex along the bolded edge $st$. Hollow vertex $\neg\ell_{i}$
will be glued to an $x$ or $\neg x$ variable vertex.}
\end{figure}

\subsection{Putting the gadgets together}

For ease, we first describe the relative case of the simplicial complex
output by the reduction. We are given a \prob{1-In-3-Sat} instance
with $m$ clauses and $n$ variables. For the variable $x$, let $m_{x}$
be the number of occurrences of $x$ or $\neg x$ as a literal in
the instance. We have assumed without loss of generality that each
$m_{x}$ is at least $2$.

For each $x$, we take $V_{x}$ to be a copy of $V(m_{x})$. The subcomplex
$R_{x}$ consists of all relativizable vertices (together with the
empty set).

For each clause $\ell_{1}\vee\ell_{2}\vee\ell_{3}$, we take both
a copy of $C^{+}$ and also a copy of $C^{\neg}$. We notice that
exactly one of $\ell_{1},\ell_{2},\ell_{3}$ is true if and only if
exactly two of $\neg\ell_{1},\neg\ell_{2},\neg\ell_{3}$ are true.
We glue the three hollow vertices of $C^{+}$ to variable vertices
corresponding to the literals $\ell_{1},\ell_{2},\ell_{3}$, and glue
the three hollow vertices of $C^{\neg}$ to the variable vertices
corresponding to the negated literals $\neg\ell_{1},\neg\ell_{2},\neg\ell_{3}$. 

Thus, for each $x$ appearing in $\ell_{1},\ell_{2},\ell_{3}$, we
have glued to both a positive variable vertex and also to a negated
variable vertex in $V_{x}$. The gluing yields a bijection between
variable vertices (taken over all $V_{x}$) and hollow vertices (taken
over all positive and negated clause gadgets).

Our relative complex thus consists of 
\[
(\Delta,\Sigma):=\dot{\bigcup_{x\,\text{variable}}}(V_{x},R_{x})\quad\cup\quad\dot{\bigcup_{\text{clauses}}}(C^{+},\bdry C^{+})\quad\cup\quad\dot{\bigcup_{\text{clauses}}}(C^{\neg},\bdry C^{\neg}).
\]

\begin{thm}
\label{thm1-IN-3ReduceRel}Given an instance $I$ of \prob{1-In-3-Sat}
with each variable appearing in at least two literals, let $(\Delta,\Sigma)$
be as constructed above. The relative simplicial complex $(\Delta,\Sigma)$
is partitionable if and only if $I$ is satisfiable.
\end{thm}

It is easy to extend this to non-relative simplicial complexes: given
$\Delta$ and $\Sigma$ as above, we let $\Delta^{+}$ be obtained
from $\Delta$ by adding a disconnected triangle, gluing an independent
copy of the triangulated $5$-vertex Möbius strip to each vertex of
$\Sigma$, and gluing an independent copy of the tetrahedral boundary
to each edge of $\Sigma$.
\begin{cor}
\label{cor:1-IN-3Reduce} Given an instance $I$ of \prob{1-In-3-Sat}
with each variable appearing in at least two literals, let $\Delta^{+}$
be as constructed above. The simplicial complex $\Delta^{+}$ is partitionable
if and only if $I$ is satisfiable.
\end{cor}

It is clear that the reduction may be implemented in polynomial time.
Theorem~\ref{thm:Main} is an immediate consequence of Corollary~\ref{cor:1-IN-3Reduce}.

\subsection{Size of the reduction}

We count the triangles in $\Delta$ and in $\Delta^{+}$. We note
that $C^{+}$ has $9$ vertices ($6$ of them relative), $18$ edges
($6$ relative), and $10$ triangles; $C^{\neg}$ has $11$ vertices
($8$ relative), $22$ edges ($9$ relative), and $12$ triangles.

Each $V_{x}$ contributes $(m_{x}-1)\cdot12$ triangles to $\Delta$,
while each clause contributes $10+12$ triangles in $C^{+}$ and $C^{\neg}$
to $\Delta$. Noting that $\sum m_{x}=3m$, the relative complex $(\Delta,\Sigma)$
has $58m-12n$ triangles.

Finally, each vertex of $\Sigma$ contributes $5$ additional triangles
to $\Delta^{+}$ (using the $5$-triangle Möbius strip), each edge
contributes $4$ additional triangles, and the empty set contributes
$1$ additional triangle. Thus, each variable $x$ contributes $(m_{x}-1)\cdot5$
additional triangles, and each clause contributes $14\cdot5+15\cdot4$
triangles. We see that the simplicial complex $\Delta^{+}$ has $1+203m-17n$
triangles.

\section{\label{sec:PartitionWGluing}Partitioning complexes with gluing}

The complexes $(\Delta,\Sigma)$ and $\Delta^{+}$ of Theorem~\ref{thm1-IN-3ReduceRel}
and Corollary~\ref{cor:1-IN-3Reduce} are built by gluing together
gadgets along uncomplicated subcomplexes. In this section, we develop
language to describe each ``piece'' of such a partitioning scheme.
The approach here will offer better control of gluing along small
intersections as compared to e.g. \cite[Theorem 3.1]{Duval/Goeckner/Klivans/Martin:2016}.

Let $(\Kappa,\Sigma)$ be a relative simplicial complex, and let $\Gamma$
be another subcomplex of $\Kappa$. It is convenient to require that
no face of $\Gamma$ be a facet of $\Kappa$. We wish to think of
$\Kappa$ as sitting inside a larger (relative) simplicial complex
as a full-dimensional subcomplex, and $\Gamma$ as where $\Kappa$
is glued to the rest of the complex.

We say that a \emph{partitioning with gluing at $\Gamma$} of $(\Kappa,\Sigma)$
is a map $\varphi$ from facets of $(\Kappa,\Sigma)$ to faces of
$(\Kappa,\Sigma)$ where $\varphi(\sigma)\subseteq\sigma$ and so
that 
\begin{enumerate}
\item The sets of faces $[\varphi(\sigma),\sigma]$ are pairwise disjoint.
\item Each face of $(\Kappa,\Sigma\cup\Gamma)$ is in exactly one of the
$[\varphi(\sigma),\sigma]$.
\end{enumerate}
We see that each face of $(\Gamma,\Gamma\cap\Sigma)$ is in at most
one of the $[\varphi(\sigma),\sigma]$. If a face $\gamma$ of $\Gamma$
is in one of the $[\varphi(\sigma),\sigma]$, then we say it is \emph{used};
otherwise it is \emph{unused}. 

Thus, if a $d$-dimensional complex may be written as $\Kappa\cup\Lambda$,
where $\Kappa$ and $\Lambda$ are both $d$-dimensional and $\Kappa\cap\Lambda$
is lower dimensional, then a partitioning of $\Kappa\cup\Lambda$
induces a partitioning with gluing at $\Kappa\cap\Lambda$ of $\Kappa$,
and also a partitioning with gluing at $\Kappa\cap\Lambda$ of $\Lambda$.
Given partitionings with gluing of $\Kappa$ and of $\Lambda$, we
may be able to recover a partitioning of $\Kappa\cup\Lambda$ if every
face of the intersection is used in one or the other of $\Kappa$
or $\Lambda$.

Indeed, the following follows immediately from definition:
\begin{prop}
\label{prop:PartitioningFromGluing}Let $\Kappa$ be a pure $d$-dimensional
simplicial complex, where $\Kappa$ is the union of pure $d$-dimensional
complexes $\Kappa=\Kappa_{1}\cup\Kappa_{2}\cup\cdots\cup\Kappa_{n}$;
let $\Sigma$ be a subcomplex of $\Kappa$, and assume that each facet
is in exactly one $\Kappa_{i}$. If each $(\Kappa_{i},\Kappa_{i}\cap\Sigma)$
admits a partitioning with gluing $\varphi_{i}$ at $\Kappa_{i}\cap(\bigcup_{j\neq i}\Kappa_{j})$
in such a way that each gluing face not in $\Sigma$ is used by exactly
one $\varphi_{i}$, then $\bigcup\varphi_{i}$ forms a partitioning
of $(\Kappa,\Sigma)$. 
\end{prop}

\begin{rem}
Conversely, in the situation of Proposition~\ref{prop:PartitioningFromGluing},
a partitioning of $(\Kappa,\Sigma)$ restricts to partitionings with
gluing of the constituent $(\Kappa_{i},\Kappa_{i}\cap\Sigma)$'s with
the ``used exactly once'' property.
\end{rem}

We are interested in putting restrictions on the used faces of a partitioning
with gluing of an uncomplicated simplicial complex. We give several
examples, which we will make good use of later.
\begin{lem}
\label{lem:PwGempty}Let $\Kappa$ be a $d$-dimensional simplex for
$d\geq1$ and let $\Gamma$ be $\left\{ \emptyset\right\} $. The
unique partitioning with gluing at $\Gamma$ of $\Kappa=(\Kappa,\emptyset)$
uses the face $\emptyset$.
\end{lem}

\begin{proof}
The simplex $\Kappa$ has $2^{d+1}$ faces, and any partitioning with
gluing that does not use $\emptyset$ covers at most $2^{d}$ faces.
\end{proof}
\begin{lem}
\label{lem:PwGcyl} Let $\Kappa$ be a Möbius strip or cylinder with
an equal number of vertices and triangles. Let $\Gamma$ be generated
by a single vertex $v$ of $\Kappa$. Every partitioning with gluing
at $\Gamma$ of $(\Kappa,\{\emptyset\})$ uses $v$.
\end{lem}

\begin{proof}
Note that $(\Kappa,\{\emptyset\})$ is partitionable by \cite[Propositions 4.3, 4.4]{Santamaria-Galvis:2022}.
Let $f_{1}=f_{3}$ be the number of vertices/triangles of $\Kappa$.
Since the Euler characteristic $\chi(\Kappa)=0$, the number of edges
is $f_{2}=2f_{1}$. But now if $\varphi$ is a partitioning with gluing,
then $[\varphi(\sigma),\sigma]$ covers $2$ edges if $\varphi(\sigma)$
is a vertex, and $1$ or $0$ edges otherwise. The result follows
by counting.
\end{proof}
\begin{lem}
\label{lem:PwGbdryS2}Let $\Kappa$ be a triangulation of the $2$-sphere,
and let $\Gamma$ be generated by a single edge $e$ of $\Kappa$.
If $\varphi$ is a partitioning with gluing at $\Gamma$ of $(\Kappa,\{\emptyset\})$,
then either 
\begin{enumerate}[label=(\alph*)]
\item  there is a facet $\sigma$ with $\varphi(\sigma)=\sigma$, or else
\item the edge $e$ is used, but its endpoint vertices are not.
\end{enumerate}
\end{lem}

\begin{proof}
Let $f_{1}$, $f_{2}$, and $f_{3}$ be the number of vertices, edges,
and triangles of $\Kappa$. Then we must have at least $f_{1}-2$
vertices in the image of $\varphi$. This leaves $(2,f_{2}-2f_{1}+4,f_{3}-f_{1}+2)$
vertices, edges, and triangles. Since the Euler characteristic $\chi(\Kappa)=2$,
we have $f_{2}-2f_{1}+4=f_{3}-f_{1}+2$. Thus, if we do not have a
facet $\sigma$ with $\varphi(\sigma)=\sigma$, then every remaining
edge (including $e$) must be in the image of $\varphi$, and no additional
vertices may be in the image.
\end{proof}
\begin{rem}
The boundary of a tetrahedron is easily seen to admit both types of
partitioning with gluing at an edge from Lemma~\ref{lem:PwGbdryS2}.
\end{rem}

As a consequence of these examples, we obtain:
\begin{prop}
\label{prop:RelToAbs}Let $\Delta$ be a pure $2$-dimensional simplicial
complex, and $\Sigma$ be a subcomplex. Let $\Delta^{+}$ be obtained
by adding a disconnected triangle, gluing a copy of a 5-vertex Möbius
strip at each vertex of $\Sigma$, and gluing a copy of the tetrahedral
boundary to each edge of $\Sigma$. 

If $\Sigma$ is $0$-dimensional or $\left\{ \emptyset\right\} $,
or if $\Sigma$ is $1$-dimensional with $\chi(\Delta,\Sigma)=0$,
then $(\Delta,\Sigma)$ is partitionable if and only if $\Delta^{+}$
is partitionable.
\end{prop}

\begin{proof}
If $\Sigma$ is $0$-dimensional or $\{\emptyset\}$, then the result
is immediate from Proposition~\ref{prop:PartitioningFromGluing},
the induced partitionings with gluing, Lemmas~\ref{lem:PwGempty}
and \ref{lem:PwGcyl}, and \cite[Proposition 4.3]{Santamaria-Galvis:2022}.
The empty set is used by the disconnected triangle, and each vertex
is used by its Möbius strip.

If $\Sigma$ is $1$-dimensional, then for Lemma~\ref{lem:PwGbdryS2}
to say each edge is used by its tetrahedral boundary, we need also
that $\tilde{\chi}(\Delta^{+})=0$. (See Lemma~\ref{lem:PartEulerChar}.)
Now $\Delta^{+}$ is the disjoint union of the relative complex $(\Delta,\Sigma)$,
the disconnected triangle, copies of a relative Möbius strip $(M,\{\emptyset\})$,
and copies of a relative sphere $(S^{2},S^{0})$. Since each (relative)
complex in the union has $0$ Euler characteristic, so does $\Delta^{+}$.
\end{proof}
\begin{rem}
\label{rem:SwitchingComplex}The \emph{switching complex} desired
by Kleinschmidt and Onn \cite{Kleinschmidt/Onn:1995,Kleinschmidt/Onn:1996}
may be formulated in this language: it is a simplicial complex $\Kappa$,
with gluing subcomplex $\Gamma=\{\emptyset,v_{1},v_{2},v_{3}\}$ such
that any partitioning with gluing at $\Gamma$ of $(\Kappa,\{\emptyset\})$
uses either all or none of $v_{1},v_{2},v_{3}$, and so that both
types of partitionings with gluing exist.

The variable gadgets of this paper have a similar rigidity with respect
to partitionings with gluing, so long as they are embedded in a simplicial
complex with trivial Euler characteristic. 
\end{rem}

\section{\label{sec:MainProofs}Proof of the main theorems}

Throughout this section, let $(\Delta,\Sigma)$ be as in Theorem~\ref{thm1-IN-3ReduceRel},
and $\Delta^{+}$ be as in Corollary~\ref{cor:1-IN-3Reduce}.

In Section~\ref{subsec:RelEulerChar}, we will show that the relative
Euler characteristic $\chi(\Delta,\Sigma)$ is zero. In Section~\ref{subsec:PartitioningGadgets},
we will show that this highly constrains partitionings with gluing
of our gadgets: in particular, an appropriate partitioning with gluing
of the variable gadget $(V_{x},R_{x})$ will use all positive or all
negated variable vertices, corresponding to the two truth values of
$x$. In Section~\ref{subsec:AssembledProof} we will assemble these
ingredients into a proof of Theorem~\ref{thm1-IN-3ReduceRel}.

\subsection{\label{subsec:RelEulerChar}Relative Euler characteristic}

We begin by computing the relative Euler characteristic of $(\Delta,\Sigma)$,
which will be useful in light of Lemma~\ref{lem:PartEulerChar}. 

Since $V^{0}$ is obtained by gluing together two spheres and a disc
along contractible intersections, we have $\tilde{\chi}(V^{0})=2$.
Since $C^{+}$ and $C^{\neg}$ are contractible, $\bdry C^{+}$ is
an $S^{1}$, and $\bdry C^{\neg}$ is a bouquet of two $S^{1}$'s,
it follows that $\chi(C^{+},\bdry C^{+})=1$ and $\chi(C^{\neg},\bdry C^{\neg})=2$. 

As $V_{x}$ is homotopy equivalent to a bouquet of $m_{x}-1$ copies
of $V^{0}$, each variable gadget $(V_{x},R_{x})$ contributes $(2m_{x}-2)-(m_{x}-2)=m_{x}$
to the relative Euler characteristic. Gluing $(C^{+},\bdry C^{+})$
along three vertices contributes $1-3=-2$ to the relative Euler characteristic,
and gluing $(C^{\neg},\bdry C^{\neg})$ along three vertices contributes
$2-3=-1$. Here, we have already counted the hollow vertices of $C^{+}$
and $C^{\neg}$ in the $V_{x}$ gadgets. Our Euler characteristic
of the relative complex is 
\[
\chi(\Delta,\Sigma)=\left(\sum m_{x}\right)+\sum_{\text{clauses}}(-2-1)=3m-3m=0.
\]

\subsection{\label{subsec:PartitioningGadgets}Partitioning the gadgets}

The essential ingredient in the proof is the $V^{0}$ relative simplicial
complex.

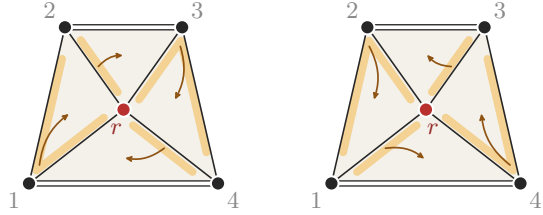
\begin{figure}
% The two states of a copy gadget.  Arrow from R_F into F; band = faces of the interval.
\begin{tikzpicture}[scale=1.25]
  % ---- state {1,3}: facet h i i+1 owns spoke h i  ("clockwise"; x true)
  \copycoords{a}{0}
  \copyedges{a}
  \vcov{a-1}{a-h}{a-2}
  \ecov{a-h}{a-2}{a-3}
  \vcov{a-3}{a-h}{a-4}
  \ecov{a-h}{a-4}{a-1}
  \node[vA] at (a-h) {};
  \node[v,label={[glab]below left:1}] at (a-1) {};
  \node[v,label={[glab]above left:2}] at (a-2) {};
  \node[v,label={[glab]above right:3}] at (a-3) {};
  \node[v,label={[glab]below right:4}] at (a-4) {};
  \node[glab,text=relred!80!black] at ($(a-h)+(250:0.22)$) {$r$};
  % ---- state {2,4}: facet h i i+1 owns spoke h(i+1)  ("counterclockwise"; x false)
  \copycoords{b}{3.2}
  \copyedges{b}
  \vcov{b-2}{b-1}{b-h}
  \ecov{b-3}{b-h}{b-2}
  \vcov{b-4}{b-3}{b-h}
  \ecov{b-1}{b-h}{b-4}
  \node[vA] at (b-h) {};
  \node[v,label={[glab]below left:1}] at (b-1) {};
  \node[v,label={[glab]above left:2}] at (b-2) {};
  \node[v,label={[glab]above right:3}] at (b-3) {};
  \node[v,label={[glab]below right:4}] at (b-4) {};
  \node[glab,text=relred!80!black] at ($(b-h)+(290:0.22)$) {$r$};
\end{tikzpicture}

\caption{\label{fig:V0partitionings}Partitioning with gluing of $V^{0}$ using
$1,3$ (left) and using $2,4$ (right). The arrows are drawn from
$\varphi(\sigma)$ to $\sigma$; shaded bands indicate the edges in
$[\varphi(\sigma),\sigma]$.}

\end{figure}

\begin{lem}
\label{lem:PwG-V0}If $\varphi$ is a partitioning of $(V^{0},\{\emptyset,r\})$
with gluing at $\Gamma=\left\{ \emptyset,1,2,3,4\right\} $, then
either
\begin{enumerate}[label=(\alph*)]
\item  there is a facet $\sigma$ with $\varphi(\sigma)=\sigma$, or 
\item $\varphi$ uses vertices $1$ and $3$ and does not use vertices $2$
or $4$, or else
\item $\varphi$ uses vertices $2$ and $4$ and does not use vertices $1$
or $3$.
\end{enumerate}
\end{lem}

\begin{proof}
By Lemma~\ref{lem:PwGbdryS2}, if there is no facet $\sigma$ with
$\varphi(\sigma)=\sigma$, then the restrictions of $\varphi$ to
the two tetrahedral boundaries use the edges $23$ and $14$. By counting
edges, $\varphi$ must send two facets to vertices and two to edges.
The result now follows by inspection. See Figure~\ref{fig:V0partitionings}.
\end{proof}
Since $V_{x}$ is obtained by gluing together many copies of $V^{0}$,
we obtain immediately the following.
\begin{cor}
\label{cor:PwG-Vx}If $\varphi$ is a partitioning of $(V_{x},R_{x})$
with gluing at the subcomplex $\Gamma$ generated by variable vertices,
then either
\begin{enumerate}[label=(\alph*)]
\item  there is a facet $\sigma$ with $\varphi(\sigma)=\sigma$, or 
\item $\varphi$ uses all of the positive variable vertices and no negated
variable vertices, or else
\item $\varphi$ uses all of the negated variable vertices and no positive
variable vertices.
\end{enumerate}
\end{cor}

\begin{proof}
If (a) does not occur, then by Lemma~\ref{lem:PwG-V0} the vertices
shared by adjacent copies of $V^{0}$ are either all used by the predecessor
copy of $V^{0}$ or all by the successor copy.
\end{proof}
We now consider the clause gadgets. Let $\Tau$ be the simplicial
complex with vertices $1,2,3,\ell$ and triangles $12\ell,23\ell,13\ell$;
let $\Sigma_{\Tau}$ be the subcomplex generated by $12$ and $13$.
Thus, the relative complexes $(C^{+},\bdry C^{+})$ respectively $(C^{\neg},\bdry C^{\neg})$
are formed by gluing together three copies of $(\Tau,\Sigma_{\Tau})$
around a triangle respectively around three triangles sharing a common
edge.
\begin{lem}
\label{lem:PwGtau}If $\varphi$ is a partitioning of $(\Tau,\Sigma_{\Tau})$
with gluing at the subcomplex $\Gamma$ generated by $\ell$ and $23$,
then either
\begin{enumerate}[label=(\alph*)]
\item  there is a facet $\sigma$ with $\varphi(\sigma)=\sigma$, or
\item $\varphi$ uses both $\ell$ and $23$, or else
\item $\varphi$ uses neither $\ell$ nor $23$.
\end{enumerate}
\end{lem}

\begin{proof}
If there is no facet with $\varphi(\sigma)=\sigma$, then count edges!
If $\ell$ is not used, then the three edges $1\ell$, $2\ell$, $3\ell$
are covered by three facets. If $\ell$ is used, then $[\ell,\sigma]$
covers two edges (each containing $\ell$). For the remaining two
facets, we need both of the remaining two non-relative edges, including
$23$.
\end{proof}
\begin{lem}
\label{lem:PwGposclause}If $\varphi$ is a partitioning of $(C^{+},\bdry C^{+})$
with gluing at $\{\emptyset,\ell_{1},\ell_{2},\ell_{3}\}$, then either 
\begin{enumerate}[label=(\alph*)]
\item  there is a facet $\sigma$ with $\varphi(\sigma)=\sigma$, or
\item exactly two of $\ell_{1},\ell_{2},\ell_{3}$ are used by $\varphi$.
\end{enumerate}
\end{lem}

\begin{proof}
Assume that there is no facet with $\varphi(\sigma)=\sigma$. The
partitioning $\varphi$ matches the center triangle to one of its
edges (as its vertices are relative). By three applications of Lemma~\ref{lem:PwGtau},
$\varphi$ does not use $\ell_{i}$ in the corresponding copy of $\Tau$,
but does use $\ell_{j}$ in the other two copies.
\end{proof}
Figure~\ref{fig:CplusEx} shows such a partitioning with gluing of
$(C^{+},\bdry C^{+})$ with $\ell_{1}$ not used. 

The corresponding result for $C^{\neg}$ is proved in exactly the
same way.
\begin{lem}
\label{lem:PwGnegclause}If $\varphi$ is a partitioning of $(C^{\neg},\bdry C^{\neg})$
with gluing at $\{\emptyset,\neg\ell_{1},\neg\ell_{2},\neg\ell_{3}\}$,
then either 
\begin{enumerate}[label=(\alph*)]
\item  there is a facet $\sigma$ with $\varphi(\sigma)=\sigma$, or
\item exactly one of $\neg\ell_{1},\neg\ell_{2},\neg\ell_{3}$ is used by
$\varphi$.
\end{enumerate}
\end{lem}

All of the possibilities for partitionings with gluing enumerated
in the lemmas and corollary of this section can indeed be constructed.
\begin{prop}
There are partitionings with gluing satisfying each enumerated type
in Corollary~\ref{cor:PwG-Vx} and Lemmas~\ref{lem:PwG-V0}, \ref{lem:PwGtau},
\ref{lem:PwGposclause}, and \ref{lem:PwGnegclause}.
\end{prop}

\begin{proof}
Inspection.
\end{proof}

\subsection{\label{subsec:AssembledProof}Proof of Theorem~\ref{thm1-IN-3ReduceRel}}

We are now ready to prove Theorem~\ref{thm1-IN-3ReduceRel}. 

First, suppose that $\varphi$ is a partitioning of $(\Delta,\Sigma)$.
Then $\varphi$ induces a partitioning with gluing of each variable
and clause gadget:
\begin{itemize}
\item In each variable gadget $(V_{x},R_{x})$, we have an induced partitioning
with gluing at the subcomplex generated by the variable vertices.
\item In each positive clause gadget $(C^{+},\bdry C^{+})$ and negated
clause gadget $(C^{\neg},\bdry C^{\neg})$, we have an induced partitioning
with gluing at the subcomplex generated by the three hollow vertices.
\end{itemize}
Since the Euler characteristic $\chi(\Delta,\Sigma)$ is zero, Lemma~\ref{lem:PartEulerChar}
tells us that no facet $\sigma$ has $\varphi(\sigma)=\sigma$. It
follows from Corollary~\ref{cor:PwG-Vx} that each $(V_{x},R_{x})$
uses either all positive or all negated variable vertices. Setting
$x$ true if $(V_{x},R_{x})$ uses the positive variable vertices
(and otherwise false) yields a truth assignment $\mathcal{T}$.

The remaining variable vertices must be used by the clause gadgets
(both positive and negated) corresponding to clauses containing $x$
or $\neg x$ as a literal. By Lemmas~\ref{lem:PwGposclause} and
\ref{lem:PwGnegclause}, each $(C^{+},\bdry C^{+})$ must use exactly
two of these variable vertices left over from $\mathcal{T}$ and each
$(C^{\neg},\bdry C^{\neg})$ must use exactly one. Since $C^{+}$
attaches to the literals of the associated clause and $C^{\neg}$
to the negations of these literals, we see that $\mathcal{T}$ is
a legal \prob{1-In-3-Sat} truth assignment.

Conversely, if $\mathcal{T}$ is a \prob{1-In-3-Sat} truth assignment,
then we put together partitionings with gluing of each gadget using
Proposition~\ref{prop:PartitioningFromGluing}. Since $\chi(\Delta,\Sigma)=0$,
we will not want any facets with $\varphi(\sigma)=\sigma$ by Lemma~\ref{lem:PartEulerChar}.

We take a partitioning with gluing of each $(V_{x},R_{x})$ that uses
$x$ or $\neg x$ according to whether $x$ is true or false in $\mathcal{T}$.
Since $\mathcal{T}$ has exactly one true literal in each clause,
we can use the remaining variable vertices in $(C^{+},\bdry C^{+})$
and $(C^{\neg},\bdry C^{\neg})$ gadgets. See Figure~\ref{fig:CplusEx}.

Corollary~\ref{cor:1-IN-3Reduce} is an immediate consequence of
Theorem~\ref{thm1-IN-3ReduceRel} and Proposition~\ref{prop:RelToAbs},
together with the Euler characteristic computation of Section~\ref{subsec:RelEulerChar}.

\begin{figure}
% A state of the positive clause: l_1 true, l_2, l_3 false.
\begin{tikzpicture}[scale=1.45]
  \clausecoords\clausebase
  % Cl_1 takes the slot pr fed by the true literal
  \ecov{r}{p}{q}
  % converter 1 (true literal): empty state, rotation of spokes
  \ecov{u1}{p}{r}\ecov{u1}{w1}{p}\ecov{u1}{r}{w1}
  % converters 2,3 (false literals): take u and their slot
  \vcov{u2}{w2}{p}\ecov{u2}{q}{w2}\ecov{p}{q}{u2}
  \vcov{u3}{w3}{q}\ecov{u3}{r}{w3}\ecov{q}{r}{u3}
  \node[voutc,label={[lab,label distance=0.5pt]60:$\ell_1$}] at (u1) {};
  \node[vused,label={[lab,label distance=1pt]15:$\ell_2$}] at (u2) {};
  \node[vused,label={[lab,label distance=1pt]165:$\ell_3$}] at (u3) {};
\end{tikzpicture}

\caption{\label{fig:CplusEx}Partitioning with gluing of $C^{+}$ associated
with truth assignment $\ell_{1}=\text{true}$, $\ell_{2}=\ell_{3}=\text{false}$.
The arrows are drawn from $\varphi(\sigma)$ to $\sigma$; shaded
bands indicate the edges in $[\varphi(\sigma),\sigma]$. Vertices
$\ell_{2}$ and $\ell_{3}$ are used by this gadget.}
\end{figure}
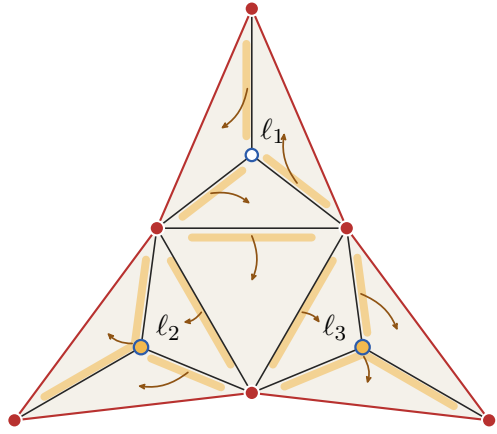

\begin{rem}
\label{rem:DupedVars}The reduction assumes that a \prob{1-In-3-Sat}
instance has every variable occurring in at least two literals of
clauses. It is possible to avoid this assumption, using a single vertex
relative to the empty set as a $V(1)$ gadget. As this approach would
require more care with the partition gluing arguments, we do not follow
it further.
\end{rem}

\section{\label{sec:Remarks}Additional remarks and questions}

Partitionability has often been studied in the context of Cohen-Macaulay
simplicial complexes \cite{Stanley:1996}. Indeed, it was a long-standing
conjecture disproved by Duval, Goeckner, Klivans and Martin \cite{Duval/Goeckner/Klivans/Martin:2016}
that every Cohen-Macaulay simplicial complex is partitionable. I do
not know the answer to the following question:
\begin{question}
\label{que:CMPartNPcomplete}Is \Partitionability{} NP-complete when
restricted to Cohen-Macaulay complexes? 
\end{question}

Examples of Cohen-Macaulay simplicial complexes that are not partitionable
seem to be relatively difficult to construct. In \cite{Duval/Goeckner/Klivans/Martin:2016},
the authors construct a $3$-dimensional complex with 42 tetrahedra.
No examples seem to be known of $2$-dimensional complexes that are
Cohen-Macaulay but not partitionable (as noted in \cite{Duval/Goeckner/Klivans/Martin:2016}),
and for example \cite{Santamaria-Galvis:2022} shows that every triangulation
of the dunce cap or of $\rr\mathrm{P}^{2}$ is partitionable.
\begin{question}
\label{que:CMPartDim2}Is every $2$-dimensional Cohen-Macaulay complex
also partitionable?
\end{question}

Of course, a ``yes'' answer to Question~\ref{que:CMPartDim2} would
mean that Question~\ref{que:CMPartNPcomplete} is only interesting
in dimensions $3$ and higher.

In contrast to Questions~\ref{que:CMPartNPcomplete} and \ref{que:CMPartDim2},
the problem of \prob{Shellability} is NP-complete even for Cohen-Macaulay,
partitionable, $2$-dimensional complexes. Indeed, the construction
of \cite{Santamaria-Galvis/Woodroofe:2021} is built up from gadgets.
The union of choice gadgets (corresponding to variables) is shellable,
hence partitionable. The constraint gadgets ($\mathrm{T}^{(2)}$ and
$\mathrm{T}^{(3)}$, in the notation of that paper) are partitionable
relative to the subcomplex along which they are attached, as may be
checked by computer. Proposition~\ref{prop:PartitioningFromGluing}
then yields partitionability; the Cohen-Macaulay property follows
by repeated application of \cite[Lemma 2.2]{Duval/Goeckner/Klivans/Martin:2016}.

The full conjecture of Kleinschmidt and Onn \cite{Kleinschmidt/Onn:1995},
that \Partitionability{} is NP-complete already for $2$-dimensional
pseudomanifolds, also remains open. In a $2$-dimensional pseudomanifold,
every edge lies in at most two triangles. Our construction here has
edges that lie in exactly three triangles in two places: the edges
where we attach tetrahedral boundaries, and an edge in the $C^{\neg}$
gadget.

\bibliographystyle{9_Users_russw_Documents_Research_mypapers_Partitionability_is_NP-complete_hamsplain}
\bibliography{8_Users_russw_Documents_Research_mypapers_Partitionability_is_NP-complete_Master}

\end{document}